\documentclass{amsproc}
\usepackage{amssymb}
\usepackage{amsfonts}

\theoremstyle{plain}

\newtheorem{definition}{Definition}

\newtheorem{lemma}{Lemma}

\newtheorem{proposition}{Proposition}
\newtheorem{remark}{Remark}

\newtheorem{theorem}{Theorem}
\numberwithin{equation}{section}
\input{tcilatex}

\begin{document}
\title[Proportional Sidon sets]{Sidon sets are proportionally Sidon with
small Sidon constants}
\author{Kathryn E. Hare}
\address{Dept. of Pure Mathematics\\
University of Waterloo\\
Waterloo, Ont. Canada N2L 3G1}
\email{kehare@uwaterloo.ca}
\thanks{This research was supported in part by NSERC grant 2016-03719.}
\author{Robert (Xu) Yang}
\address{Dept. of Pure Mathematics\\
University of Waterloo\\
Waterloo, Ont. Canada N2L 3G1}
\email{yangxu\_robert@hotmail.com}
\thanks{This paper is in final form and no version of it will be submitted
for publication elsewhere.}

\begin{abstract}
In his seminal work on Sidon sets, Pisier found an important
characterization of Sidonicity: A set is Sidon if and only if it is
proportionally quasi-independent. Later, it was shown that Sidon sets were
proportionally `special' Sidon in several other ways. Here, we prove that
Sidon sets in torsion-free groups are proportionally $n$-degree independent,
a higher order of independence than quasi-independence, and we use this to
prove that Sidon sets are proportionally Sidon with Sidon constants
arbitrarily close to one, the minimum possible value.
\end{abstract}

\maketitle

\section{Introduction}

Let $G$ be a compact abelian group and $\Gamma $ its discrete dual. A subset 
$E\subseteq \Gamma $ is called a Sidon set if there is a constant $C$ such
that every bounded $E$-function $\phi $ is the restriction of the Fourier
Stieltjes transform of a finite measure on $G$ of measure norm at most $%
C\left\Vert \phi \right\Vert _{\infty }$. The least such $C$ is called the
Sidon constant of $E$. Sidon sets are well known to be plentiful. Indeed,
infinite examples can be found in every infinite subset of $\Gamma $ and
include lacunary sets (in $\Gamma =\mathbb{Z}$) and independent sets.

Sidon sets have been extensively studied, yet fundamental questions remain
open. As the class of Sidon sets is closed under finite unions, it is
natural to ask whether every Sidon set is the finite union of a `nicer',
i.e., more restricted, class of interpolation sets. Important progress on
this general problem was made when Pisier \cite{P1} characterized Sidon sets
as those which are `proportionally' quasi-independent (special Sidon sets
that are independent-like). Later Ramsey \cite{Ra}, proved that Sidon sets
are proportionally $I_{0}$ (special Sidon sets where the interpolating
measure can be chosen to be discrete) in a uniform sense and subsequently
one of the authors with Graham \cite{GH5} showed that they are
proportionally $\varepsilon $-Kronecker (special Sidon sets defined by an
approximation property) under the assumption that $\Gamma $ has no elements
of finite order.

In this paper we prove that if $\Gamma $ has no elements of finite order,
then every Sidon set is proportionally Sidon with Sidon constants
arbitrarily close to one. This will be established by generalizing Pisier's
proportional quasi-independent characterization of Sidon to higher degrees
of independence. Of course, every Sidon set has Sidon constant at least one
and this is the Sidon constant for an independent set in the case that $%
\Gamma $ has no elements of finite order. But many groups, including $%
\mathbb{Z}$, have no non-trivial independent sets or even any subsets with
Sidon constant equal to one, other than one and two element sets.

Our proportionality result does not hold, in general, for groups that admit
elements of finite order as there are such groups with the property that the
Sidon constant of every non-trivial independent set is bounded away from
one. However, we do show the proportionality result holds when $\Gamma $ is
a product of finite groups with prime order tending to infinity.

\section{Definitions and Basic properties}

We begin by recalling some well-known equivalent definitions of Sidon. For
proofs of these facts and other properties of Sidon sets mentioned below, we
refer the reader to \cite{GHbook} or \cite{LR}.

\begin{definition}
A subset $E\subseteq \Gamma $ is called a \textbf{Sidon set} if whenever $%
\phi :E\rightarrow \mathbb{C}$ is a bounded function, there is a measure $%
\mu $ on $G$ with the property that $\widehat{\mu }(\gamma )=\phi (\gamma )$
for every $\gamma \in E$ and $\left\Vert \mu \right\Vert _{M(G)}\leq
C\left\Vert \phi \right\Vert _{\infty }$. The least such constant $C$ is
known as the \textbf{Sidon constant} of $E$. The set $E$ is called $\mathbf{I%
}_{0}$ if the measure $\mu $ can be chosen to be discrete.
\end{definition}

\begin{proposition}
\label{equiv}The following are equivalent:

\begin{enumerate}
\item $E$ is Sidon.

\item There are constants $C$ and $0\leq \delta <1$ such that for every $%
\phi \in $Ball$(\ell ^{\infty }(E))$ there is a measure $\mu $ on $G$ with $%
\left\Vert \mu \right\Vert _{M(G)}\leq C$ and satisfying 
\begin{equation*}
\sup_{\gamma \in E}\{\left\vert \phi (\gamma )-\widehat{\mu }(\gamma
)\right\vert \}\leq \delta .
\end{equation*}

\item For every $\phi :E\rightarrow \{\pm 1\}$ there is a measure $\mu $ on $%
G$ such that 
\begin{equation*}
\sup_{\gamma \in E}\{\left\vert \phi (\gamma )-\widehat{\mu }(\gamma
)\right\vert \}<1.
\end{equation*}

\item There is a constant $C$ such that whenever $f$ is a trigonometric
polynomial with supp$\widehat{f}\subseteq E,$ then 
\begin{equation*}
\sum_{\gamma \in E}\left\vert \widehat{f}(\gamma )\right\vert \leq C\sup
\{\left\vert f(x)\right\vert :x\in G\}.
\end{equation*}
\end{enumerate}
\end{proposition}

The minimal constant $C$ satisfying (4) is (also) the Sidon constant. An
iterative argument can be given to show that if item (2) is satisfied, then
the Sidon constant of $E$ is at most $C/(1-\delta )$. In particular, an
independent set in a group with no elements of finite order has Sidon
constant $1$. From (3) one can immediately see that independent sets in
groups with elements of finite order are also Sidon.

Finite sets $F$ are always Sidon with Sidon constant at most $\sqrt{%
\left\vert F\right\vert }$. In $\mathbb{Z}$, one and two element sets have
Sidon constant one, but this is never the case for sets of three or more
elements, see \cite{Ne}. A classical example of an infinite Sidon set is the
subset $E=\{3^{n}\}_{n=1}^{\infty }\subseteq $ $\mathbb{Z}$. Indeed, given $%
\phi \in \ell ^{\infty }(E)$ with $\left\Vert \phi \right\Vert _{\infty
}\leq 1/2,$ we can take as the interpolating measure the Riesz product
measure%
\begin{equation*}
\mu =\prod_{j=1}^{\infty }\left( 1+2\mathcal{R}(\phi
(3^{n})e^{i3^{n}x})\right)
\end{equation*}%
where the infinite product notation means $\mu $ is the weak $\ast $ limit
in $M(\mathbb{T)}$. As $\left\Vert \mu \right\Vert _{M(\mathbb{T)}}=1$, the
Sidon constant of $E$ is at most $2$. In fact, the set $\{3^{n}\},$ or more
generally any lacunary set $\{n_{j}\}\subseteq \mathbb{Z}^{+}$ (meaning $%
\inf n_{j+1}/n_{j}=q>1$), is an example of an $I_{0}$ set (although the
Riesz product measure argument does not show this).

The class of $I_{0}$ sets is a proper subset of the Sidon sets since the
class of Sidon sets is closed under finite unions, but the class of $I_{0}$
sets is not. But $I_{0}$ sets are also plentiful and every infinite subset
of $\Gamma $ contains an infinite $I_{0}$ set. It is of interest to
understand the relationship between Sidon and $I_{0}$ sets since $I_{0}$
sets are known to not cluster at any continuous character in the Bohr
topology, while it is unknown whether Sidon sets can (even) be dense in the
Bohr compactification of $\Gamma $.

Another interesting class of Sidon sets are the $\varepsilon $-Kronecker
sets: $E$ $\subseteq \Gamma $ is $\varepsilon $-Kronecker if for every $\phi
:E\rightarrow \mathbb{T}$ there exists $x\in G$ such that $\left\vert \phi
(\gamma )-\gamma (x)\right\vert <\varepsilon $ for all $\gamma \in E$. Any
lacunary set $\{n_{j}\}$ with $\inf n_{j+1}/n_{j}>2$ is $\varepsilon $%
-Kronecker for some $\varepsilon <2$ and every set that is $(2-\varepsilon )$%
-Kronecker is Sidon \cite{HR}. There are examples of Sidon sets that are not 
$(2-\varepsilon )$-Kronecker for some groups $\Gamma ,$ but it is unknown if
such examples can be found in $\mathbb{Z}$.

A weakened version of independence is the following notion.

\begin{definition}
Let $n\in \mathbb{N}$. We say that $E\subseteq \Gamma $ is $n$\textbf{%
-degree independent} if whenever $k\in \mathbb{N}$, $\gamma _{1},...,\gamma
_{k}$ are distinct elements in $E$ and $m_{1},..,m_{k}$ are integers with $%
\left\vert m_{i}\right\vert \leq n$, then $\prod_{i=1}^{k}\gamma
_{i}^{m_{i}}=\mathbf{1}$ implies $\gamma _{i}^{m_{i}}=\mathbf{1}$ for all $%
i=1,...,k$ where $\mathbf{1}$ denotes the identity in $\Gamma $. A $1$%
-degree independent set is usually called \textbf{quasi-independent} and a $%
2 $-degree independent set is called \textbf{dissociate}.

The set $E$ is said to be $n$\textbf{-length independent} if whenever $%
\gamma _{1},...,\gamma _{n}$ are distinct elements in $E$ and $%
m_{1},..,m_{n}\in \{0,\pm 1\}$ then $\prod_{i=1}^{n}\gamma _{i}^{m_{i}}=%
\mathbf{1}$ implies $\gamma _{i}^{m_{i}}=\mathbf{1}$ for all $i$.
\end{definition}

\begin{remark}
Note that the condition $\gamma _{i}^{m_{i}}=\mathbf{1}$\textbf{\ }can be
replaced with $\gamma _{i}=\mathbf{1}$ in the case that all $\gamma _{i}$
have order greater than $n$.
\end{remark}

Clearly $n$-degree independent implies $n$-length independent and a set is
independent if and only if it is $n$-degree independent for every $n$. The
set $E=\{3^{n}\}$ is a dissociate set and a Riesz product construction shows
that any dissociate set is Sidon. A modification of this argument can be
given to show that quasi-independent sets are also Sidon.

Significant efforts have been made to characterize Sidon sets in terms of
these more restricted classes of sets. Towards this end, Malliavin \cite{Ma}
showed that any Sidon set not containing the identity, in the group $\oplus 
\mathbb{Z}_{p}$ for $p$ prime, is a finite union of independent sets, while
Bourgain \cite{Bo} proved that every Sidon set $E\subseteq \Gamma \backslash
\{\mathbf{1\}}$ is a finite union of $n$-length independent sets. However,
it is unknown if every Sidon set is:

\begin{itemize}
\item a finite union of $I_{0}$ sets

\item a finite union of $\varepsilon $-Kronecker sets

\item a finite union of quasi-independent sets
\end{itemize}

Pisier introduced probabilistic techniques to study these and related
questions and obtained important `proportional' characterizations of Sidon
sets (see Theorem \ref{Pisier} below). These characterizations inspired a
number of other such characterizations and is the motivation for this paper.
Here is a sampling of these `proportional' characterizations.

\textbf{Terminology:} Given two classes of sets $\mathcal{A}$, $\mathcal{B}$%
, we will say that $E\in \mathcal{A}$ is proportionally $\mathcal{B}$ if
there is some constant $\delta >0$ such that for every finite $F\subseteq E$
there is some $H\subseteq F$ such that $\left\vert H\right\vert \geq \delta
\left\vert F\right\vert $ and $H\in \mathcal{B}$.

\begin{theorem}
\label{Pisier}(a) The following are equivalent for $E\subseteq \Gamma
\backslash \{\mathbf{1}\}$:

\begin{enumerate}
\item $E$ is Sidon;

\item $E$ is proportionally quasi-independent;

\item There exists a constant $C$ such that $E$ is proportionally Sidon with
Sidon constant at most $C;$

\item There exists an integer $M$ such that $E$ is proportionally $I_{0}(M)%
\footnote{$E$ is $I_{0}(M)$ if for every $\phi \in $Ball$(\ell ^{\infty
}(E)) $ there is a discrete measure $\mu =\sum_{j=1}^{M}a_{j}\delta _{x_{j}}$
with $\left\vert a_{j}\right\vert \leq 1$ and $\sup_{\gamma \in E}\left\vert
\phi (\gamma )-\widehat{\mu }(\gamma )\right\vert \leq 1/2$.}.$
\end{enumerate}

(b) If $\Gamma $ has only finitely many elements of order $2^{k}$ for any $k$
and $E$ has no elements of order two, then $E$ is Sidon if and only if $E$
is proportional $\varepsilon $-Kronecker for some $\varepsilon <\sqrt{2}$.
\end{theorem}

The equivalence of (1 - 3) is a deep result of Pisier, see \cite{P1}-\cite%
{P3}, with later proofs given by Bourgain in \cite{B1}, \cite{B2}. The
equivalence of (4) was shown by Ramsey in \cite{Ra}, while (b) was
established in \cite{GH5} along with other related proportional
equivalences. We refer the reader also to \cite[ch. 7,9]{GHbook} and \cite[%
p. 482-499]{LQ} for expositions of these results.

In this paper, we will modify Pisier's technique to prove that if $\Gamma $
has no elements of finite order, then $E$ is Sidon if and only if $E$ is
proportionally $n$-degree independent for each $n$ if and only for every
constant $C>1,$ $E$ is proportionally Sidon with Sidon constant $C$. Partial
results are obtained in the case that $\Gamma $ has elements of finite order.

\section{Proportional Sidon subsets in torsion-free groups}

In this section our main focus will be on torsion-free, discrete abelian
groups $\Gamma ,$ groups which have no elements of finite order. These are
the groups whose dual groups $G$ are connected. We will first extend
Pisier's proportional quasi-independent characterization of Sidon to $n$%
-degree independence, and then use this to deduce that Sidon sets are
proportionally Sidon with constants arbitrarily close to $1$.

We begin with some preliminary lemmas that hold for general discrete abelian
groups.

\begin{lemma}
\label{1}Suppose $E\subseteq \Gamma \diagdown \{\mathbf{1}\}$ is Sidon.
There is a constant $K$, depending only on the Sidon constant of $E,$ such
that for all finite subsets $A\subseteq E$ and real numbers $(a_{\gamma
})_{\gamma \in A},$ we have 
\begin{equation*}
\int \exp \left( \sum_{\gamma \in A}a_{\gamma }\mathcal{R}(\gamma )\right)
\leq \exp \left( K\sum_{\gamma \in A}\left\vert a_{\gamma }\right\vert
^{2}\right) .
\end{equation*}
\end{lemma}

\begin{proof}
This is well known and is a straightforward argument using the power series
expansion of the exponential function and the fact that if $E$ is a Sidon
set with Sidon constant $S$, then%
\begin{equation*}
\left\Vert f\right\Vert _{p}\leq 2S\sqrt{p}\left\Vert f\right\Vert _{2}
\end{equation*}%
for any integer $p\geq 2$ and trigonometric polynomial $f$ with supp$%
\widehat{f}\subseteq E$. The details are left to the reader.
\end{proof}

\textbf{Notation}: Given $E\subseteq \Gamma $ and $k\in \mathbb{N},$ we let 
\begin{equation*}
E_{k}=\{\gamma ^{k}:\gamma \in E\}.
\end{equation*}

\begin{lemma}
\label{2}Let $n\in \mathbb{N}$ and assume $\Gamma $ contains no non-trivial
elements of order $\leq n$. Suppose $E\subseteq \Gamma \diagdown \{\mathbf{1}%
\}$ and $E_{k}$ is Sidon for each $k=1,...,n$. Then there is a constant $%
K_{n},$ depending only on $n$ and the Sidon constants of the sets $E_{k},$ $%
k=1,...,n,$ such that for all $0<\lambda <1/n$ and finite subsets $%
A\subseteq E$, we have%
\begin{equation*}
\int \prod_{\gamma \in A}\left( 1+\lambda \sum_{k=1}^{n}\mathcal{R(}\gamma
^{k})\right) \leq \exp \left( K_{n}\left\vert A\right\vert n^{3}\lambda
^{2}\right) .
\end{equation*}%
(Here $\left\vert A\right\vert $ denotes the cardinality of the set $A$.)
\end{lemma}

\begin{proof}
Let $A\subseteq E$ be finite. Since $\left\vert \sum_{k=1}^{n}\mathcal{R(}%
\gamma ^{k})\right\vert \leq n$ and $\lambda <1/n$, we have 
\begin{equation*}
\prod_{\gamma \in A}\left( 1+\lambda \sum_{k=1}^{n}\mathcal{R(}\gamma
^{k})\right) \leq \exp \left( \lambda \sum_{\gamma \in A}\sum_{k=1}^{n}%
\mathcal{R(}\gamma ^{k})\right) .
\end{equation*}%
Put $A^{(n)}=\bigcup_{k=1}^{n}A_{k}$. We can write 
\begin{equation*}
\sum_{\gamma \in A}\sum_{k=1}^{n}\mathcal{R(}\gamma ^{k})=\sum_{\beta \in
A^{(n)}}a_{\beta }\mathcal{R}(\beta ).
\end{equation*}%
Note that the coefficients $a_{\beta }$ satisfy $0\leq a_{\beta }\leq 2n$
since the assumption that $\Gamma $ contains no elements of order $\leq n$
ensures that $\mathcal{R(}\gamma ^{k})=\mathcal{R(}\chi ^{k})$ for $\gamma
,\chi \in A$ and $k\leq n$ only if $\gamma =\chi $ or $\overline{\chi }$.

Since a finite union of Sidon sets is Sidon with Sidon constant depending
only on the Sidon constants of the individual sets and the number of sets in
the union, $A^{(n)}$ is Sidon with Sidon constant depending only on that of
the sets $E_{k}$ and $n$. Thus Lemma \ref{1} and the fact that $\left\vert
A^{(n)}\right\vert \leq n\left\vert A\right\vert $ implies that there is a
constant $k_{n}$ with 
\begin{eqnarray*}
\int \exp \left( \lambda \sum_{\gamma \in A}\sum_{k=1}^{n}\mathcal{R(}\gamma
^{k})\right) &=&\int \exp \left( \lambda \sum_{\beta \in A^{(n)}}a_{\beta }%
\mathcal{R}(\beta )\right) \\
&\leq &\exp \left( k_{n}\sum_{\beta \in A^{(n)}}\lambda ^{2}a_{\beta
}^{2}\right) \leq \exp \left( 4k_{n}n^{3}\lambda ^{2}\left\vert A\right\vert
\right) .
\end{eqnarray*}
\end{proof}

Next, we upgrade Pisier's proportional quasi-independent characterization of
Sidon to $n$-degree independent proportional sets. Our proof follows his
strategy.

\begin{proposition}
\label{qi}Let $n\in \mathbb{N}$ and assume $\Gamma $ contains no non-trivial
elements of order $\leq n$. Suppose $E\subseteq \Gamma \diagdown \{\mathbf{1}%
\}$ and $E_{k}$ is Sidon for each $k=1,...,n$. There exists $\delta _{n}>0$
such that for each finite set $F\subseteq E$ there is a further finite
subset $H\subseteq F$ which is $n$-degree independent and satisfies $%
\left\vert H\right\vert \geq \delta _{n}\left\vert F\right\vert .$
\end{proposition}

\begin{proof}
Fix an integer $n$. For a finite subset $F\subseteq E$, let $\mathcal{C}%
_{n}(F)$ denote the cardinality of the set 
\begin{equation*}
\{(\xi _{\gamma })_{\gamma \in F}\in \{0,\pm 1,...,\pm n\}^{F}:\prod_{\gamma
\in F}\gamma ^{\xi _{\gamma }}=1\}.
\end{equation*}

The first step of the proof is to show that there are constants $\delta
=\delta _{n},\alpha =\alpha _{n}>0$ such that for each finite $F\subseteq E,$
there is a further subset $H\subseteq F$ with $\left\vert H\right\vert \geq
\delta \left\vert F\right\vert $ and $\mathcal{C}_{n}(H)\leq 2\cdot
2^{\alpha \left\vert H\right\vert }$. To see this, fix such $F$ and let $%
\lambda \in (0,1/n)$. Let $(\varepsilon _{\gamma })_{\gamma \in F}$ be a
collection of independent $0,1$-valued random variables on a probability
space $(\Omega ,\mathbb{P})$ such that $\mathbb{P}\{\varepsilon _{\gamma
}=1\}=\lambda /2$. An application of Fubini's theorem, independence and
Lemma \ref{2} gives%
\begin{eqnarray*}
\mathbb{E}\int \prod_{\gamma \in F}\left( 1+\varepsilon _{\gamma
}\sum_{k=1}^{n}(\gamma ^{k}+\gamma ^{-k})\right) &=&\int \mathbb{E}%
\prod_{\gamma \in F}\left( 1+\varepsilon _{\gamma }\sum_{k=1}^{n}(\gamma
^{k}+\gamma ^{-k})\right) \\
&=&\int \prod_{\gamma \in F}\left( 1+\lambda \sum_{k=1}^{n}\mathcal{R(}%
\gamma ^{k})\right) \\
&\leq &\exp \left( K_{n}n^{3}\lambda ^{2}\left\vert F\right\vert \right) .
\end{eqnarray*}%
If we let $F(\omega )=\{\gamma \in F:\varepsilon _{\gamma }(\omega )=1\},$
then%
\begin{equation*}
\mathbb{E}(\mathcal{C}_{n}(F(\omega )))=\mathbb{E}\int \prod_{\gamma \in
F}\left( 1+\varepsilon _{\gamma }\sum_{k=1}^{n}(\gamma ^{k}+\gamma
^{-k})\right) \leq \exp \left( K_{n}n^{3}\lambda ^{2}\left\vert F\right\vert
\right) .
\end{equation*}%
By Markov's inequality, with probability at least $1/2$ we have 
\begin{equation*}
\mathcal{C}_{n}(F(\omega ))\leq 2\exp \left( K_{n}n^{3}\lambda
^{2}\left\vert F\right\vert \right) .
\end{equation*}

Notice that if $\gamma _{1}\neq \gamma _{2}$, then $\mathbb{E((\varepsilon }%
_{\gamma _{1}}-\mathbb{E\varepsilon }_{\gamma _{1}})(\varepsilon _{\gamma
_{2}}-\mathbb{E\varepsilon }_{\gamma _{2}}))=0$, thus%
\begin{eqnarray*}
\mathbb{E}\left( \left\vert F(\omega )\right\vert -\mathbb{E}\left\vert
F(\omega )\right\vert \right) ^{2} &=&\mathbb{E}\left( \sum_{\gamma \in
F}(\varepsilon _{\gamma }-\mathbb{E\varepsilon }_{\gamma })\right) ^{2} \\
&=&\sum_{\gamma \in F}\mathbb{E}(\varepsilon _{\gamma }-\mathbb{E\varepsilon 
}_{\gamma })^{2} \\
&=&\left\vert F\right\vert (\lambda /2-\lambda ^{2}/4)\leq \left\vert
F\right\vert \lambda /2.
\end{eqnarray*}%
Since, also, $\mathbb{E(}\left\vert F(\omega )\right\vert )=\left\vert
F\right\vert \lambda /2$, it follows from Chebyshev's inequality that 
\begin{eqnarray*}
\mathbb{P\{}\left\vert F(\omega )\right\vert &\leq &\left\vert F\right\vert
\lambda /4\}\leq \mathbb{P}\left\{ \left( \left\vert F(\omega )\right\vert -%
\mathbb{E}\left\vert F(\omega )\right\vert \right) ^{2}\geq \left\vert
F\right\vert ^{2}\lambda ^{2}/16\right\} \\
&\leq &\frac{\left\vert F\right\vert \lambda /2}{\left\vert F\right\vert
^{2}\lambda ^{2}/16}=\frac{8}{\left\vert F\right\vert \lambda }.
\end{eqnarray*}

Choose $\lambda =\lambda _{n}>0$ so small that $\exp (4K_{n}n^{3}\lambda )<2$
and let $\alpha \in (0,1)$ be given by $2^{\alpha }=\exp (4K_{n}n^{3}\lambda
)$. With this choice of $\lambda $, $\mathbb{P\{}$ $\left\vert F(\omega
)\right\vert >\left\vert F\right\vert \lambda /4\}>1/2$ if $\left\vert
F\right\vert $ is sufficiently large, and for any such $\omega $, 
\begin{equation*}
2\cdot 2^{\alpha \left\vert F(\omega )\right\vert }=2\exp
(4K_{n}n^{3}\lambda \left\vert F(\omega )\right\vert )\geq 2\exp
(K_{n}n^{3}\lambda ^{2}|F|).
\end{equation*}%
Thus $\left\vert F(\omega )\right\vert >\left\vert F\right\vert \lambda /4$
and $\mathcal{C}_{n}(F(\omega ))\leq 2\cdot 2^{\alpha \left\vert F(\omega
)\right\vert }$ with positive probability.

This proves that there are constants $\delta =\lambda /4$ and $0<\alpha <1$
such that for any finite $F\subseteq E$ there is a subset $H=F(\omega
)\subseteq E$ with $\left\vert H\right\vert \geq \delta \left\vert
F\right\vert $ and $\mathcal{C}_{n}(H)\leq 2\cdot 2^{\alpha \left\vert
H\right\vert }$.

We will say a finite set $A\subseteq \Gamma $ is an $n$-relation set if
there exists $(\xi _{\gamma })_{\gamma \in A}\in \{\pm 1,...,\pm n\}^{A}$
with $\prod_{\gamma \in A}\gamma ^{\xi _{\gamma }}=1$. Given $A,$ we let $%
\mathcal{M}(A)$ denote a maximal (with respect to inclusion) subset of $A$
that is an $n$-relation set. The maximality ensures that $A\diagdown 
\mathcal{M}(A)$ is an $n$-degree independent set. To complete the proof of
the proposition, we will establish the following:

\textit{Claim}: Given a sufficiently large finite set $F$ satisfying $%
\mathcal{C}_{n}(F)\leq 2\cdot 2^{\alpha \left\vert F\right\vert }$ for some $%
\alpha >0,$ there exists a constant $0<\theta <1,$ depending only on $\alpha 
$, and a subset $H_{1}\subseteq F$ with $\left\vert H_{1}\right\vert \geq
\left\vert F\right\vert /2$ and having $\left\vert \mathcal{M}%
(H_{1})\right\vert \leq \theta \left\vert H_{1}\right\vert $.

Of course, in this case $H=H_{1}\diagdown \mathcal{M}(H_{1})$ is an $n$%
-degree independent subset of $F$ with cardinality at least $(1-\theta
)\left\vert F\right\vert /2$.

A technical fact we will use in proving the claim is that if we let 
\begin{equation*}
s(\theta )=\left( \frac{1-\theta }{2}\right) \log _{2}\left( \frac{2e}{%
1-\theta }\right) \text{ for }\theta \in (0,1),
\end{equation*}%
then, since $\binom{n}{k}\leq \left( \frac{ne}{k}\right) ^{k}$, we have 
\begin{equation*}
\binom{n}{\frac{n(1-\theta )}{2}}\leq 2^{s(\theta )n}.
\end{equation*}

Assume the claim is false. Then whenever $H_{1}$ is a subset of $F$ with $%
\left\vert H_{1}\right\vert =\left\vert F\right\vert /2$, (without loss of
generality we can assume $F$ has an even number of elements) we must have $%
\left\vert \mathcal{M}(H_{1})\right\vert >\theta \left\vert H_{1}\right\vert 
$.

As $\lim_{\theta \rightarrow 1}s(\theta )=0,$ we can choose $\theta $
sufficiently close to $1$ that $1-s(\theta )>\alpha $. A combinatorial
argument shows that if $H_{0}\subseteq F$ and $\theta \left\vert
F\right\vert /2<\left\vert H_{0}\right\vert <\left\vert F\right\vert /2$,
then the number of subsets $H_{1}\subseteq F$ containing $H_{0}$ and having
cardinality $\left\vert F\right\vert /2$ is 
\begin{eqnarray}
\left\vert \{H_{1}\subseteq F:\left\vert H_{1}\right\vert =\left\vert
F\right\vert /2,H_{1}\supseteq H_{0}\}\right\vert  &\leq &\binom{\left\vert
F\right\vert -\left\vert H_{0}\right\vert }{\left\vert F\right\vert
/2-\left\vert H_{0}\right\vert }  \notag \\
&\leq &\binom{\left\vert F\right\vert }{\left\vert F\right\vert (1-\theta )/2%
}\leq 2^{s(\theta )\left\vert F\right\vert }\text{.}  \label{bound}
\end{eqnarray}

We let $\mathcal{F}$ denote the collection of all subsets $H_{0}\subseteq F$
such that there exists $H_{1}\subseteq F$ with $\left\vert H_{1}\right\vert
=\left\vert F\right\vert /2$ and $\mathcal{M}(H_{1})=H_{0}$. Of course, $%
\mathcal{C}_{n}(F)\geq \left\vert \mathcal{F}\right\vert $. Thus%
\begin{eqnarray*}
\binom{\left\vert F\right\vert }{\left\vert F\right\vert /2} &=&\left\vert
\{H_{1}\subseteq F:\left\vert H_{1}\right\vert =\left\vert F\right\vert
/2\}\right\vert \\
&=&\sum_{H_{0}\subseteq F}\left\vert \{H_{1}\subseteq F:\left\vert
H_{1}\right\vert =\left\vert F\right\vert /2,\mathcal{M}(H_{1})=H_{0}\}%
\right\vert \\
&\leq &\sum_{H_{0}\in \mathcal{F}}\left\vert \{H_{1}\subseteq F:\left\vert
H_{1}\right\vert =\left\vert F\right\vert /2,H_{1}\supseteq
H_{0}\}\right\vert \\
&\leq &\left\vert \mathcal{F}\right\vert 2^{\left\vert F\right\vert s(\theta
)}\leq \mathcal{C}_{n}(F)2^{\left\vert F\right\vert s(\theta )},
\end{eqnarray*}%
where the second inequality comes from (\ref{bound}). Since $1-s(\theta
)>\alpha ,$ this implies 
\begin{equation*}
\mathcal{C}_{n}(F)\geq \binom{\left\vert F\right\vert }{\left\vert
F\right\vert /2}2^{-\left\vert F\right\vert s(\theta )}\sim \frac{1}{\sqrt{%
\left\vert F\right\vert }}2^{\left\vert F\right\vert }2^{-\left\vert
F\right\vert s(\theta )}>2\cdot 2^{\alpha \left\vert F\right\vert }
\end{equation*}%
if $\left\vert F\right\vert $ is sufficiently large, and that is a
contradiction.
\end{proof}

\begin{lemma}
\label{div}Assume $\Gamma $ is a torsion-free group. If $E\subseteq \Gamma $
is a Sidon set, then for all positive integers $n,$ the set $E_{n}=\{\gamma
^{n}:\gamma \in E\}$ is also a Sidon set with the same Sidon constant as $E$.
\end{lemma}

\begin{proof}
Assume $f(x)=\sum_{\gamma \in E}a_{\gamma }\gamma ^{n}(x)$ is a
trigonometric polynomial with supp$\widehat{f}\subseteq E_{n}$. Choose $%
x_{0}\in G$ such that $\left\vert \sum_{\gamma \in E}a_{\gamma }\gamma
(x_{0})\right\vert =\left\Vert \sum a_{\gamma }\gamma \right\Vert _{\infty }$
and pick $y\in G$ such that $y^{n}=x_{0}$. (We can do this since $\Gamma $
torsion-free implies $G$ is a divisible group.) As $\gamma ^{n}(y)=\gamma
(x_{0}),$ 
\begin{equation*}
\left\vert f(y)\right\vert =\left\vert \sum_{\gamma \in E}a_{\gamma }\gamma
(x_{0})\right\vert =\left\Vert \sum_{\gamma \in E}a_{\gamma }\gamma
\right\Vert _{\infty }\geq \frac{1}{S}\sum_{\gamma \in E}\left\vert
a_{\gamma }\right\vert ,
\end{equation*}%
where $S$ is the Sidon constant of $E$. Hence $E_{n}$ is a Sidon set with
constant at most $S$.

It is even easier to see that the Sidon constant of $E$ is at most the Sidon
constant of $E_{n},$ hence we have equality.
\end{proof}

We are now ready to prove our main result.

\begin{theorem}
\label{main}Assume $\Gamma $ is a torsion-free group. The following are
equivalent for $E\subseteq \Gamma \backslash \{\mathbf{1}\}$:

\begin{enumerate}
\item $E$ is Sidon;

\item For each positive integer $n$, $E$ is proportionally $n$%
-degree-independent;

\item For each $\varepsilon >0,$ $E$ is proportionally Sidon with Sidon
constant at most $1+\varepsilon .$
\end{enumerate}
\end{theorem}

\begin{proof}
The fact that (2) and (3) each imply (1) is a consequence of Pisier's
proportional characterizations Theorem \ref{Pisier}.

The fact that (1) implies (2) follows directly from the previous Lemma and
Prop. \ref{qi}.

We turn now to the proof that (1) implies (3). Fix $\varepsilon >0$ and
choose $\eta >0$ so that $(1-\eta )/(1+\eta )\geq 1/(1+\varepsilon )$. Pick $%
n$ such that $\left\vert e^{2\pi it}-1\right\vert <\eta /2$ on $[-1/n,1/n]$
and consider the continuous, even function $f$ $:\mathbb{T=[-}1/2,1/2]%
\mathbb{\rightarrow R}$ given by $f(x)=n-n\left\vert x\right\vert $ for $%
x\in \lbrack -1/n,1/n]$ and $f(x)=0$ otherwise. Obviously, $f\geq 0$ and $%
\widehat{f}(0)=\left\Vert f\right\Vert _{1}=1$. An easy calculation shows $%
\widehat{f}(\pm 1)\geq 1-\eta /2$.

Select an even, real-valued trigonometric polynomial $q$ such that $%
\left\Vert f-q\right\Vert _{\infty }<\eta /2$ and let $p$ be the even,
positive, trigonometric polynomial given by 
\begin{equation}
p=\frac{q+\eta /2}{\widehat{q}(0)+\eta /2}.  \label{poly}
\end{equation}%
This normalization ensures that $\widehat{p}(0)=1$ and $\widehat{p}(\pm
1)\geq (1-\eta )/(1+\eta )\geq 1/(1+\varepsilon )$. Let $N$ be the degree of 
$p$.

Since Sidon sets are proportionally $n$-degree independent for each $n$,
there exists $\delta >0$ such that each finite $F\subseteq E$ admits an $%
(N+1)$-degree independent subset $H$ with $\left\vert H\right\vert \geq
\delta \left\vert F\right\vert $.

We now give a Riesz product construction to bound the Sidon constant of $H$.
Let $\phi :H\rightarrow \mathbb{C}$ with $\left\Vert \phi \right\Vert
_{\infty }\leq 1/(1+\varepsilon )$. Let $u_{\gamma }=\phi (\gamma
)/\left\vert \phi (\gamma )\right\vert \mathbb{\ }$be a complex number of
modulus one, and define the trigonometric polynomial $P_{\gamma }$ on $G$ by 
\begin{equation*}
P_{\gamma }(x)=\frac{\left\vert \phi (\gamma )\right\vert }{\widehat{p}(1)}%
\sum_{n=-N}^{N}\widehat{p}(n)(u_{\gamma }\gamma (x))^{n}+1-\frac{\left\vert
\phi (\gamma )\right\vert }{\widehat{p}(1)}\text{ for }x\in G.
\end{equation*}%
Then 
\begin{equation*}
\widehat{P_{\gamma }}(\mathbf{1})=\frac{\left\vert \phi (\gamma )\right\vert 
}{\widehat{p}(1)}\widehat{p}(0)+1-\frac{\left\vert \phi (\gamma )\right\vert 
}{\widehat{p}(1)}=1,
\end{equation*}%
\begin{equation*}
\widehat{P_{\gamma }}(\gamma )=\frac{\left\vert \phi (\gamma )\right\vert }{%
\widehat{p}(1)}\widehat{p}(1)u_{\gamma }=\phi (\gamma )
\end{equation*}%
and the degree of $P_{\gamma }$ is $N$. Since $\left\vert \phi (\gamma
)\right\vert /\widehat{p}(1)\leq 1$, $P_{\gamma }\geq 0$ and therefore $%
\left\Vert P_{\gamma }\right\Vert _{1}=1$.

Let $P=\prod_{\gamma \in H}P_{\gamma }$. Since $H$ is $(N+1)$-degree
independent, standard arguments show that $\left\Vert P\right\Vert _{1}=%
\widehat{P}(\mathbf{1})=1$ and $\widehat{P}(\gamma )=\phi (\gamma )$ for all 
$\gamma \in H$. As $\left\Vert \phi \right\Vert \leq 1/(1+\varepsilon )$,
this proves that $H$ is a Sidon set with Sidon constant bounded by $%
1+\varepsilon ,$ as we desired to show.
\end{proof}

\begin{remark}
(i) An antisymmetric Sidon set that has the additional property that the
interpolating measure can always be chosen to be positive is called a
Fatou-Zygmund set with the Fatou-Zymund constant defined in the obvious way.
As the Riesz product measure $P$ constructed in the proof of the Theorem is
a positive measure, we actually have shown that any Sidon set in a
torsion-free group is proportionally Fatou-Zygmund with Fatou-Zygmund
constants arbitrarily close to $1$.

(ii) Since finite sets have the same Sidon and $I_{0}$ constants (\cite{Gr}%
), it also follows that $E$ is Sidon if and only if for each $\varepsilon
>0, $ $E$ is proportionally $I_{0}$ with $I_{0}$ constant at most $%
1+\varepsilon .$
\end{remark}

\section{Sidon sets in torsion groups}

When the group $\Gamma $ has elements of finite order, the situation is
quite different. In \cite{Bo}, Bourgain proved that every Sidon set in $%
\Gamma =\mathbb{Z}_{n}^{\mathbb{N}},$ where $n$ has no repeated prime
factors, is a finite union of independent sets. His methods actually show
the following.

\begin{proposition}
Suppose $\Gamma =\oplus _{i=1}^{N}\mathbb{Z}_{p_{i}}^{\mathbb{N}}$, $p_{i}$
prime and assume $p_{1}=\min \{p_{j}\}_{j=1}^{N}$. Then any Sidon set in $%
\Gamma $ is a finite union of sets that are $(p_{1}-1)$-degree independent.
\end{proposition}

However, such Sidon sets are not necessarily proportionally Sidon with Sidon
constants arbitrarily close to $1$. Indeed, it is easy to see using Prop. %
\ref{equiv}(4) that if, for example, $\Gamma =\mathbb{Z}_{p}^{\mathbb{N}}$
for a prime number $p,$ then any two element subset of $\Gamma $ (even if an
independent set) has Sidon constant bounded below by%
\begin{equation*}
\sup_{\alpha ,\beta }\left( \min_{\xi \text{ }p\text{-root unity}}\frac{%
\left\vert \alpha \right\vert +\left\vert \beta \right\vert }{\left\vert
\alpha +\beta \xi \right\vert }\right) \geq \sec (\pi /(2p).
\end{equation*}

We can, however, obtain our `small constants' proportionality result for
products of cyclic groups $\mathbb{Z}_{p_{i}}$ where $(p_{i})$ tends to
infinity.

\begin{proposition}
Suppose $\Gamma =\oplus _{i=1}^{\infty }\mathbb{Z}_{p_{i}}$ where $%
(p_{i})_{i}$ is a sequence of prime numbers tending to infinity. If $%
E\subseteq \Gamma $ is Sidon, then for all $\varepsilon >0$ there is some $%
\delta >0$ such that for all finite $F\subseteq E,$ there exists a further
finite subset $H\subseteq F$ with Sidon constant bounded by $1+\varepsilon $
and satisfying $\left\vert H\right\vert \geq \delta \left\vert F\right\vert $%
.
\end{proposition}

\begin{proof}
Fix $\varepsilon >0$ and suppose $F$ is a finite subset of $E$. Let $p$ be
the polynomial defined in (\ref{poly}) and put $N=\deg p$. Choose $n_{0}$
such that $p_{i}>N+1$ for all $i>n_{0}$. Let $\Gamma _{1}=\oplus
_{i=1}^{n_{0}}\mathbb{Z}_{p_{i}}$ and $M=\left\vert \Gamma _{1}\right\vert $%
. Choose $F_{1}\subseteq F$ such that $F_{1}=\gamma Y$ where $\gamma \in
\Gamma _{1}$, $Y\subseteq \oplus _{i>n_{0}}\mathbb{Z}_{p_{i}}$ and $%
\left\vert F_{1}\right\vert \geq \left\vert F\right\vert /M$. Since
translation preserves Sidon constants, $Y$ is a Sidon set with constant at
most that of $E$.

Now consider $Y_{k}=\{\chi ^{k}:\chi \in Y\}$ for $k\leq N$. Since the
elements of $\mathbb{Z}_{p_{i}}$ for $i>n_{0}$ have prime order exceeding $N$%
, essentially the same argument as in the proof of Lemma \ref{div} shows
that each $Y_{k}$ is Sidon with Sidon constant the same as $E$.

Applying Prop. \ref{qi} we see there is a constant $\delta >0$ (depending on 
$N)$ and an $(N+1)$-degree independent set $Y_{0}\subseteq Y$ such that $%
\left\vert Y_{0}\right\vert \geq \delta \left\vert Y\right\vert $. For $%
Y_{0},$ being $(N+1)$-degree independent is the same as saying $%
\prod_{i=1}^{k}\gamma _{i}^{m_{i}}=\mathbf{1}$ for $\left\vert
m_{i}\right\vert \leq N+1$ only if $\gamma _{i}=\mathbf{1}$ for all $i$.
That fact allows us to apply the Riesz product construction of the proof of
Theorem \ref{main} (with the polynomial $p$) and as in that proof we deduce
that the Sidon constant of $Y_{0}$ is at most $1+\varepsilon $. Of course,
this is also a bound on the Sidon constant of $H=\gamma Y_{0}$ and this
subset of $F$ has cardinality at least $(\delta /M)\left\vert F\right\vert $%
, completing the proof.
\end{proof}

\end{document}